\documentclass[12pt,reqno]{amsart}
\usepackage{amsmath}
\usepackage{amssymb}
\usepackage{amsthm}
\usepackage[left=2.7cm,top=2.7cm,right=2.7cm,bottom=2.7cm]{geometry}
\usepackage{enumerate}
\usepackage[usenames]{color}
\usepackage{graphicx}
\usepackage{hyperref}
\usepackage{caption} 
\newtheorem{theorem}{Theorem}[section]
\newtheorem{lemma}[theorem]{Lemma}

\newtheorem{claim}[theorem]{Claim}

\newtheorem{corollary}[theorem]{Corollary}

\theoremstyle{remark}
\newtheorem{remark}{Remark}

\def\bx{{\bf x}}

\newcommand{\brac}[1]{\left(#1\right)}

\newcommand{\set}[1]{\left\{#1\right\}}

\def\E{\mbox{{\bf E}}}

\def\Pr{\mbox{{\bf Pr}}}

\newcommand{\ignore}[1]{}

\allowdisplaybreaks

  \def\d{\delta} 
\def\e{\varepsilon}    
  
\def\E{\mbox{{\bf E}}}

\def\Pr{\mbox{{\bf Pr}}}

\begin{document}

\title{Loose paths in random ordered hypergraphs}

\author{Andrzej Dudek}
\address{Department of Mathematics, Western Michigan University, Kalamazoo, MI}
\email{\tt andrzej.dudek@wmich.edu}
\thanks{The first author was supported in part by a grant from the Simons Foundation MPS-TSM-00007551.}

\author{Alan Frieze}
\address{Department of Mathematical Sciences, Carnegie Mellon University, Pittsburgh, PA}
\email{\tt frieze@cmu.edu}
\thanks{The second author was supported in part by NSF grant DMS-2341774}

\author{Wesley Pegden}
\address{Department of Mathematical Sciences, Carnegie Mellon University, Pittsburgh, PA}
\email{\tt wes@math.cmu.edu}
\thanks{The third author was supported in part by NSF grant DMS-2054503}


\begin{abstract}
We consider the length of {\em ordered loose paths} in the random $r$-uniform hypergraph $H=H^{(r)}(n, p)$. A ordered loose path is a sequence of edges $E_1,E_2,\ldots,E_\ell$ where $\max\set{j\in E_i}=\min\set{j\in E_{i+1}}$ for $1\leq i<\ell$. We establish fairly tight bounds on the length of the longest ordered loose path in~$H$ that hold with high probability.
\end{abstract}

\maketitle

\section{Introduction.}
There has been considerable work on the maximum length of paths in random graphs and hypergraphs, particularly in the graph case; see the survey by Frieze \cite{Survey} for a summary what is known on the subject. Albert and Frieze \cite{AF} considered the maximum length of a path in an orientation of $G_{n,p}$ where the edge $\set{i,j},i<j$ was always oriented from $i$ to $j$. In this paper we consider a generalization of this problem to $r$-uniform hypergraphs, where $r \geq 2$ is a fixed constant throughout.

Let $H=H^{(r)}(n, p)$ be the random $r$-uniform hypergraph on the set of vertices $[n]$ such that each $r$-tuple in $\binom{[n]}{r}$ is included as an edge with probability~$p$. Let $E_p$ denote its set of edges. Define an \emph{ordered loose path} of \emph{length}~$\ell$ in~$H$ as an increasing subsequence of vertices $v_1,v_2,\dots,v_{\ell (r-1)+1}\in [n]$ such that $\{v_1<\dots<v_r\}, \{v_r<\dots<v_{2r-1}\},\dots,\{v_{(\ell-1)(r-1)+1}<\dots<v_{\ell (r-1)+1}\}$ are the edges of $H$ (so that every pair of consecutive edges intersects in a single vertex). Let $\ell_{\max}$ be the maximal length of an ordered loose path in $H$. In the following, we discuss the likely value of $\ell_{\max}$ for varying values of $p$.

\begin{theorem}\label{thm:main1}
Let $r\ge 2$ and $\Omega(1)=p\le 1-o(1)$. Then, a.a.s.
\[
\frac{(1+o(1))n}{r -p - r(1-p)^r+ (pr+1)(1-p)^rp^{-2}} \le \ell_{\max} \le (1+o(1))n\left(\frac{1}{r} + \frac{1}{r(r-2+p^{-1})} \right).
\]
\end{theorem}

\begin{corollary}\label{cor:main1}
Let $0<p<1$ be a constant. Then, for every $\e>0$ there is an $r_0=r_0(p,\e)\in \mathbb{N}$ such that for each $r\ge r_0$, we have a.a.s.\
\[
\left(\frac{1}{r}-\e\right)n\le \ell_{\max} \le \left(\frac{1}{r}+\e\right)n.
\]
\end{corollary}

\begin{corollary}\label{cor:main2}
Let $r\ge 2$ and $p=1-o(1)$. Then, we have a.a.s.\
\[
\ell_{\max} = \frac{(1+o(1))n}{r-1}.
\]
\end{corollary}

In fact, the upper bound in these corollaries is trivial (and no upper bound from Theorem~\ref{thm:main1} is needed), since always the length of an ordered loose path is at most $\frac{n}{r-1} = n\left(\frac{1}{r} + \frac{1}{r(r-1)} \right)$. However, notice that the second term of the upper bound in Theorem~\ref{thm:main1} is always smaller than the trivial one, since $\frac{1}{r(r-2+p^{-1})} < \frac{1}{r(r-1)}$ when $p<1$.

\begin{theorem}\label{thm:main2}
Let $r\ge 2$ and $\sqrt{\log n}/n^{(r-1)/4}\ll p = o(1)$ or $\Omega((\log n)^{r-1}/n^{r-1})  = p \ll 1/n^{(r-1)/2}$. Then, a.a.s.\
\[
\ell_{\max} = \Theta(np^{1/(r-1)}).
\]
\end{theorem}
(We write $A_n\ll B_n$ (resp. $B_n\gg A_n$) if $A_n/B_n\to 0$ as $n\to\infty$.)

The next theorem fills the gap for the missing range of $p$ from Theorem~\ref{thm:main2}. Unfortunately, this statement (likely the lower bound) is not optimal.
\begin{theorem}\label{thm:main3}
Let $r\ge 2$ and $\Omega(1/n^{(r-1)/2}) = p = O(\sqrt{\log n}/n^{(r-1)/4})$. Then, a.a.s.\
\[
\Omega(np^{1/(r-1)}/(\log n)^{1/(r-1)}) =  \ell_{\max} = O(np^{1/(r-1)}).
\]
\end{theorem}

The remaining (optimal) theorems describe results for the sparse regime.
\begin{theorem}\label{thm:main4}
Let $r\ge 3$ and $\omega=\omega(n)$ be such that $1\ll \omega \ll \log n$ and 
\[
\frac{1}{n^{r-1+1/\omega}} \le p = O((\log n)^{r-1}/n^{r-1}).
\]
Then, there exists $1\ll \ell_0 = O(\log n)$ such that 
\[
\left(\frac{n}{\ell_0}\right)^{r-1+1/\ell_0} p =1
\] 
and then a.a.s.\ $\ell_{\max}=\Theta(\ell_0)$.
\end{theorem}

\begin{theorem}\label{thm:main5}
Let $r\ge 3$ and $\ell\ge 2$ be an integer. If $1/n^{r-1+1/\ell} \ll p \ll 1/n^{r-1+1/(\ell+1)}$, then a.a.s. $\ell_{\max} = \ell$. 

Furthermore, let $c>0$ be a real number and $p=c/n^{r-1+1/\ell}$. Let $X$ count the number of ordered paths of length~$\ell$. Then, $X$ converges in distribution to $Po(\lambda)$, where $\lambda = c^\ell/(\ell(r-1)+1)!$.
\end{theorem}

Observe that if in Theorem~\ref{thm:main4} we allow $\omega = \Theta(1)$, then the lower bound on $p$ is considered in Theorem~\ref{thm:main5}.
Also if $\omega = \Omega(\log n)$, then 
\[
p\ge \frac{1}{n^{r-1+1/\Omega(\log n)}} = \frac{1}{n^{r-1}} \frac{1}{n^{1/\Omega(\log n)}} = \Omega\left( \frac{1}{n^{r-1}} \right),
\]
which is contained in the range of $p$ in Theorem~\ref{thm:main4}.

\subsection{Outline of the paper}
Lemma \ref{lem:geometric} of Section \ref{pre} deals with the stochastic properties of edge length. This is followed by a technical lemma and the statement of a concentration inequality.

We then divide the proof as shown in Table~\ref{tab:structure}. We have not managed to deal with small $p$, i.e., $p=o((\log n)^{r-1}/n^{r-1})$.

\renewcommand{\arraystretch}{1.3}
\begin{table}
\begin{tabular}{|c|c|c|}
\hline
\text{\textbf{Section}}&\textbf{Bound}&\textbf{Range of }$p$ \\
\hline
\ref{sec:3}&\text{Lower} & $\sqrt{\log n}/n^{(r-1)/4}\ll p \le 1-o(1)$\\
 \ref{sec:4}&\text{Lower} & $\Omega(1/n^{(r-1)/2}) = p = O(\sqrt{\log n}/n^{(r-1)/4})$\\
\ref{sec:5}&\text{Lower} & $\Omega((\log n)^{r-1}/n^{r-1}) = p \ll 1/n^{(r-1)/2}$\\ 
\ref{sec:6}&\text{Upper} & $p=\Omega((\log n)^{r-1}/n^{r-1})$\\
\ref{sec:7}&\text{Lower/Upper}& $\frac{1}{n^{r-1+1/\omega}} \le p \le O((\log n)^{r-1}/n^{r-1})\text{ or }1 \ll \omega \ll \log n$\\
\ref{sec:8a}&\text{Lower/Upper}& $1/(n^{r-1+1/\ell}) \ll p \ll 1/(n^{r-1+1/(\ell+1)})\text{ where  }\ell=O(1)$\\
\hline
\end{tabular}
\caption{The structure of the proofs.}
\label{tab:structure}
\end{table}

\section{Preliminaries}\label{pre}
To bound the length of the longest ordered loose path from below, it is natural to consider a greedy algorithm, which extends an ordered loose path from its last vertex $i$ by a hyperedge whose least vertex is $i$ and whose ``length'' $j-i$, where $j$ is its last vertex, is as small as possible.

We will use the following lemma (about random hypergraphs of edges with common left endpoints) when considering the properties of such a greedy extension algorithm.
\begin{lemma}\label{lem:geometric}
  Let $H$ be a random $r$-hypergraph on the vertex set $\mathbb{Z}^+$ such that each edge $\{1<i_2<\dots<i_{r-1}<i_{r}\}$ appears with probability~$p$. Let $X_i$ (for $i\ge 0$) be a random variable, which equals $r+i-1$ if (I) $\{1<i_2<\dots<i_{r-1} < i_r=r+i\}$ is present in~$H$ and (II) no edge of the form $\{1<j_2<\dots<j_{r-1}<j_r \leq r+i-1\}$ is present in $H$, {and is 0 otherwise}. Define $X=\sum_{i=0}^\infty X_i$. Then:
\begin{enumerate}[(i)]

\item\label{prop:geometric:i} The expected value equals
\[
\E X=\sum_{i=0}^\infty (r+i-1) (1-p)^{\binom{r-2+i}{r-1}} \left( 1-\left(1-p \right)^{\binom{r-2+i}{r-2}} \right).
\]

\item\label{prop:geometric:ii} The expected value satisfies
\[
r -p - r(1-p)^r \le \E X \le  r -p - r(1-p)^r+ (pr+1)(1-p)^rp^{-2}
\]
and, in particular, for $p=\Omega(1)$, which is independent from $r$,
\[
\E X = r-p+o_r(1),
\]
where $o_r(1)$ approaches zero as $r$ tends to infinity.

\item\label{prop:geometric:iii} For $p=o(1)$,
\[
\E X = \Theta(1/p^{1/(r-1)}).
\]
\end{enumerate}
\end{lemma}

\begin{proof}
We divide the proof into three parts.

\medskip
\textbf{Part \eqref{prop:geometric:i}:}
For $i=0$, clearly, $\E X_0 = (r-1)q_0p_0$, where $q_0=1$ and $p_0=p$. 

Now consider $i=1$. Here the edge $\{1,\dots,r\}$ is not present (this happens with probability $q_1 =1-p_0=1-p$) and an edge $e=\{1<i_2<\dots<i_{r-1}<r+1\}$ is present. Since we have exactly $\binom{r-1}{r-2}$ choices for vertices $i_2<\dots<i_{r-1}$, the latter occurs with probability $p_1=1-\left(1-p\right)^{\binom{r-1}{r-2}}$. Thus, $\E X_{1} = r q_1 p_1$.

Assume in general that $q_i$ is the probability of the event that no edge $\{1<i_2<\dots<i_{r-1}<r+j\}$ for $0\le j\le i-1$ is present. Moreover, let $p_i$ be the probability of the event that there is an edge of the form $\{1<i_2<\dots<i_{r-1}<r+i\}$. Observe that 
\[
q_i = q_{i-1}(1-p_{i-1}) \quad\text{ and }\quad p_i=1-\left(1-p \right)^{\binom{r-2+i}{r-2}}.
\]
Thus,
\[
q_i = \prod_{j=0}^{i-1} (1-p_j)
= \prod_{j=0}^{i-1} \left(1-p \right)^{\binom{r-2+j}{r-2}}
= (1-p)^{\sum_{j=0}^{i-1} \binom{r-2+j}{r-2}}
= (1-p)^{\binom{r-2+i}{r-1}}
\]
and
\[
\E X_i = (r-1+i) (1-p)^{\binom{r-2+i}{r-1}} \left( 1-\left(1-p \right)^{\binom{r-2+i}{r-2}} \right). 
\]

Consequently,
\[
\E X=\sum_{i=0}^\infty \E X_i = \sum_{i=0}^\infty (r-1+i) (1-p)^{\binom{r-2+i}{r-1}} \left( 1-\left(1-p \right)^{\binom{r-2+i}{r-2}} \right).
\]

\medskip
\textbf{Part \eqref{prop:geometric:ii}:}
We split the sum in the expression for $\E X$ into two terms: $\E X_0+\E X_1$ and $\sum_{i=2}^\infty \E X_i$. Observe that 
\[
\E X_0 + \E X_1 = (r-1)p + r(1-p)\left(1-(1-p)^{r-1}\right) = r -p - r(1-p)^r
\]
and
\[
\sum_{i=2}^\infty \E X_i \le \sum_{i=2}^\infty (r-1+i) (1-p)^{r-2+i} = (pr+1)(1-p)^rp^{-2}.
\]
Since $\E X_0 + \E X_1 \le \E X = \E X_0 + \E X_1 + \sum_{i=2}^\infty \E X_i$, we obtain
\[
r -p - r(1-p)^r \le \E X \le  r -p - r(1-p)^r+ (pr+1)(1-p)^rp^{-2}.
\]
Finally notice that both $\lim_{r\to\infty} r(1-p)^r = 0$ and $\lim_{r\to\infty} (pr+1)(1-p)^rp^{-2} = 0$ for any $p=\Omega(1)$ which is independent from $r$.
This yields $\E X = r-p+o_r(1)$.

\medskip
\textbf{Part \eqref{prop:geometric:iii}:}
For $r=2$ observe that 
\[
\E X=\sum_{i=0}^\infty (1+i) (1-p)^{i} p= p \sum_{i=0}^\infty (1+i) (1-p)^{i} = p\cdot \frac{1}{p^2} = 1/p,
\]
as required.
Therefore, we may assume that $r\ge 3$. We split the sum into two terms:
\[
S_1=\sum_{0\le i\le (r-1)/p^{1/(r-1)}} (r-1+i) (1-p)^{\binom{r-2+i}{r-1}} \left( 1-\left(1-p \right)^{\binom{r-2+i}{r-2}} \right),
\]
and
\[
S_2=\sum_{(r-1)/p^{1/(r-1)}<i} (r-1+i) (1-p)^{\binom{r-2+i}{r-1}} \left( 1-\left(1-p \right)^{\binom{r-2+i}{r-2}} \right).
\]
Observe that
\begin{align*}
S_1 &\le \sum_{0\le i\le (r-1)/p^{1/(r-1)}} (r-1+i) \left( 1-\left(1-p \right)^{\binom{r-2+i}{r-2}} \right)\\
&\le \sum_{0\le i\le (r-1)/p^{1/(r-1)}} (r-1+i) \left( 1-\left(1-p {\binom{r-2+i}{r-2}}\right) \right)\\
&= p \sum_{0\le i\le (r-1)/p^{1/(r-1)}} (r-1+i) \binom{r-2+i}{r-2}\\
&\le p \sum_{0\le i\le (r-1)/p^{1/(r-1)}} (r-1+i) (r-2+i)^{r-2}\\ 
&\le p \cdot \left((r-1)/p^{1/(r-1)} +1\right)\left(r-1+(r-1)/p^{1/(r-1)}\right)^{r-1} = O(1/p^{1/(r-1)}).
\end{align*}
Now
\begin{align*}
S_2 &= \sum_{(r-1)/p^{1/(r-1)}<i} (r-1+i) (1-p)^{\binom{r-2+i}{r-1}} \left( 1-\left(1-p \right)^{\binom{r-2+i}{r-2}} \right)\\
&\le \sum_{(r-1)/p^{1/(r-1)}<i} (r-1+i) \exp\left\{-p\binom{r-2+i}{r-1}\right\} \left( 1-\left(1-p {\binom{r-2+i}{r-2}}\right) \right)\\
&= p\sum_{(r-1)/p^{1/(r-1)}<i} (r-1+i) \binom{r-2+i}{r-2} \exp\left\{-p\binom{r-2+i}{r-1}\right\}\\
&\le p\sum_{(r-1)/p^{1/(r-1)}<i} (r-1+i) \left(\frac{e(r-2+i)}{r-2}\right)^{r-2} \exp\left\{-p\left(\frac{r-2+i}{r-1}\right)^{r-1}\right\}\\
\end{align*}
And after the change of variables $i:=r-2+i$,
\begin{align*}
S_2
&\le p\left(\frac{e}{r-2}\right)^{r-2}\sum_{(r-1)/p^{1/(r-1)}+r-2<i} (i+1) i^{r-2} \exp\left\{-p\left(\frac{1}{r-1}\right)^{r-1} i^{r-1}\right\}\\
&\le 2p\left(\frac{e}{r-2}\right)^{r-2}\sum_{(r-1)/p^{1/(r-1)}+r-2<i} i^{r-1} \exp\left\{-p\left(\frac{1}{r-1}\right)^{r-1} i^{r-1}\right\},
\end{align*}
since $i+1\le 2i$.
Consider $f(x) = x^{r-1}\exp(-cx^{r-1})$ for $x> 0$. It is easy to check that $f$ is positive and  takes the maximum at point $1/c^{1/(r-1)}$. Furthermore, $f$ is a strictly decreasing function for $x> 1/c^{1/(r-1)}$. Furthermore,
\[
\int_0^\infty x^{r-1}\exp(-cx^{r-1}) dx= \frac{\Gamma(r/(r - 1))}{c^{r/(r-1)}(r - 1)},
\]
where $\Gamma(z)=\int_0^\infty e^{-t}t^{z-1}$ is the \emph{gamma function}
(see integral~3.326 in~\cite{tables}).
Let $c=p\left(\frac{1}{r-1}\right)^{r-1}$. Then, 
\[
1/c^{1/(r-1)} = (r-1)/p^{1/(r-1)}. 
\]
Thus,
\begin{align*}
S_2 &\le 2p\left(\frac{e}{r-2}\right)^{r-2}\sum_{(r-1)/p^{1/(r-1)}+r-2<i} i^{r-1} \exp\left\{-p\left(\frac{1}{r-1}\right)^{r-1} i^{r-1}\right\}\\
&\le 2p\left(\frac{e}{r-2}\right)^{r-2} \int_{0}^\infty x^{r-1} \exp\left\{-p\left(\frac{1}{r-1}\right)^{r-1} x^{r-1}\right\}dx\\
&= 2p\left(\frac{e}{r-2}\right)^{r-2} \frac{\Gamma(r/(r - 1))}{c^{r/(r-1)}(r - 1)}\\
&= 2p\left(\frac{e}{r-2}\right)^{r-2} \frac{(r - 1)^{r-1}\Gamma(r/(r - 1))}{p^{r/(r-1)}} = O(1/p^{1/(r-1)}).
\end{align*}
Hence, 
\[
\E X = S_1 + S_2 = O(1/p^{1/(r-1)}).
\]

It remains to show the lower bound. Observe that
\[
(1+x)^k\le 1+ \frac{kx}{1-(k-1)x} =\frac{1+x}{1-(k-1)x}\quad \text{for} \quad 0\leq x(k-1)<1 \quad \text{and}\quad k\ge 1.
\]
Thus,
\[
1-\left(1-p \right)^{\binom{r-2+i}{r-2}}
\ge 1-\left(1-\frac{p {\binom{r-2+i}{r-2}}}{1+p\left({\binom{r-2+i}{r-2}}-1\right)}\right) 
=\frac{p {\binom{r-2+i}{r-2}}}{1+ p\left({\binom{r-2+i}{r-2}}-1\right)}.
\]
Moreover, for $p=o(1)$ and $0\le i\le 1/p^{1/(r-1)}$ we get
\[
1+p\left({\binom{r-2+i}{r-2}}-1\right) \le 1+p(r-2+i)^{r-2} = O(p^{1/(r-1)}) = o(1) \le 2.
\]
Hence,
\[
1-\left(1-p \right)^{\binom{r-2+i}{r-2}} \ge p {\binom{r-2+i}{r-2}}/2.
\]

Also, since $1-x \ge e^{-2x}$ for $0\le x\le 1/2$, we obtain for $p=o(1)$ and $0\le i\le 1/p^{1/(r-1)}$
\[
(1-p)^{\binom{r-2+i}{r-1}} \ge \exp\left\{-2p \binom{r-2+i}{r-1} \right\}
\ge \exp\left\{-2p (r-2+i)^{r-1} \right\} = \Omega(1).
\] 

Therefore, using $(1-a)^b\geq 1-ab$ for $0\leq a\leq 1$ and $b>0$, we see that
\begin{align*}
\E X &\ge \Omega(p) \sum_{0\le i\le 1/p^{1/(r-1)}} (r-1+i)\binom{r-2+i}{r-2} \\
&\ge \Omega(p) \sum_{0\le i\le 1/p^{1/(r-1)}} i^{r-1}\\
&\ge \Omega(p) \frac{1}{p^{1/(r-1)}} \left( \frac{1}{p^{1/(r-1)}} \right)^{r-1} = \Omega(1/p^{1/(r-1)}),
\end{align*}
finally yielding $\E X = \Theta(1/p^{1/(r-1)})$.
\end{proof}

In Table~\ref{tab:EX} we present the specific values of this expectation in case when $p=1/2$ and $r\in\{2,\dots,10\}$. We can see that it is growing closer to $r-1/2$.

\begin{table}
\begin{tabular}{ |c|c c c c c c c c c| } 
 \hline
$r$ & 2 & 3& 4& 5& 6& 7& 8& 9& 10\\
 \hline
$\E X$  & 2 & 2.642... & 3.563...& 4.531...& 5.516...& 6.509...& 7.504...& 8.502...& 9.501...\\
\hline
\end{tabular}
\caption{$\E X$ from Lemma~\ref{lem:geometric} for $p=1/2$ and $r\in\{2,\dots,10\}$.}
\label{tab:EX}
\end{table}

\begin{remark}
For $r=2$ the random variable $i$ has the geometric distribution with probability of success $p$. Indeed, in this case $X_i=1+i$ can be viewed as $i$ failures (occurring with probability $q_i=(1-p)^i$) before the first success happens (with probability $p_i=p$) for each $i\ge 0$.
\end{remark}

We will also need an easy auxiliary result. 
\begin{claim}\label{claim:connector}
Let $p=\Omega((\log n)/n^{r-1})$ and
\[
a=
\begin{cases}
\frac{4}{p} \log n & \text{ if } r=2,\\
\left(\frac{4(r-2)^{r-2}}{p} \log n \right)^{1/(r-1)} & \text{ if } r\ge3. 
\end{cases}
\]
Then, for $r=2$ and any disjoint subsets $A_1,  A_2 \subseteq [n]$ with $|A_1|=|A_2|=a\geq 20$, with probability $1-O(n^{-20})$ there is in $H^{(2)}(n, p)$ an edge $e$ such that $|A_1\cap e| = |A_2\cap e| = 1$.
Moreover, for $r\ge 3$ and any disjoint subsets $A_1, A_2, A_3 \subseteq [n]$ with $|A_1|=|A_2|=|A_3|=a$, with probability $1-O(n^{-20})$ there is in $H^{(r)}(n, p)$ an edge $e$ such that $|A_1\cap e| = |A_2\cap e| = 1$ and $|A_3\cap e| = r-2$. 
\end{claim}
\begin{proof}
We prove the statement for $r\ge 3$. (The case $r=2$ is very similar.)
Observe that the probability that there are $A_1, A_2, A_3$ with no edge described above is at most
\begin{align*}
\binom{n}{a}^3 (1-p)^{a^2 \binom{a}{r-2}} 
&\le n^{3a} \exp\left\{-pa^2 \binom{a}{r-2}\right\}\\
&= \exp\left\{3a\log n-pa^2 \binom{a}{r-2}\right\}\\
&\le \exp\left\{3a\log n-pa^2 a^{r-2}/(r-2)^{r-2}\right\}\\
&= \exp\left\{a\left(3\log n-pa^{r-1}/(r-2)^{r-2}\right)\right\}\\
&\le \exp\left\{a(3\log n-4\log n)\right\} = n^{-a}\le n^{-20}.
\end{align*}
\end{proof}
We will need the following concentration result of Warnke~\cite{W}:
\begin{lemma}[Warnke~\cite{W}]\label{warnke}
Let $W=(W_1,W_2,\ldots,W_n)$ be a family of independent random variables with $W_i$ taking values in a set $\Lambda_i$. Let $\Omega=\prod_{i\in[n]}\Lambda_i$ and suppose that $\Gamma\subseteq \Omega$ and $f:\Omega\to{\bf R}$ are given. Suppose also that whenever $\bx,\bx'\in \Omega$ differ only in the $i$-th coordinate 
\[
|f(\bx)-f(\bx')|\leq \begin{cases}c_i&if\ \bx\in\Gamma.\\d_i&otherwise.\end{cases}
\]
If $Y=f(X)$, then for all reals $\gamma_i>0$ and $t$,
\begin{equation}\label{Wineq}
\Pr(Y\geq \E(Y)+t)\leq \exp\set{-\frac{t^2}{2\sum_{i\in[n]}(c_i+\gamma_i(d_i-c_i))^2}}+\Pr(W\notin \Gamma)\sum_{i\in [n]}\gamma_i^{-1}.
\end{equation}
\end{lemma}

\section{Lower bound for $\sqrt{\log n}/n^{(r-1)/4}\ll p \le 1-o(1)$.}\label{sec:3}
We present a lower bound by applying a greedy algorithm, that we denote by GREEDY. This will imply the lower bound in Theorem~\ref{thm:main1}, and the lower bound for the upper range of $p$ in Theorem~\ref{thm:main2}.

Let $e=\{i_1<\dots<i_r\}$ be an edge. We define the \emph{length} of $e$ as $\ell(e)=i_r-i_1$. We will greedily construct a loose path $e_1,e_2,\dots$. We start at the vertex $1$ and choose the first present edge $e_1=\{1<i_2<\dots<i_{r-1}<i_r\}$, that means, an edge of the smallest length that contains vertex~1. If there is more than one edge of minimum length containing vertex~1, we break ties arbitrarily.
Now we repeat the process starting at the vertex $i_r$ and find the first edge $e_2$ that starts at $i_r$, etc. Let $L_i,i\geq 1$ be independent copies of the random variable $X$ defined in Lemma~\ref{lem:geometric}. Now define the random variable $K$ by
\[
L_1+L_2+\dots +L_K \le n-1 \quad \text{ and }\quad L_1+L_2+\dots +L_{K+1} > n-1,
\]
In which case $K$ is the number of edges in the path constructed by GREEDY. We have
\[
(\E K)(\E X)\le n-1 \quad \text{ and }\quad (\E (K+1))(\E X)>n-1
\]
implying that 
\begin{equation}\label{XK}
\frac{n-1}{\E X}-1\leq \E K\leq \frac{n-1}{\E X}
\end{equation}

We let $W_i$ denote the set of edges of $H^{(r)}(n, p)$ that contain vertex $i$ and are contained in $[i]$. We note that the $W_i$ determine $H^{(r)}(n, p)$ and that they are an independent family. Let the random variable $Y$ denote the maximum length of an ordered path in~$H^{(r)}(n, p)$, so that $Y$ is determined by $W_1,W_2,\ldots,W_n$. Clearly, $Y\ge K$ and $\E Y \ge \E K$.
We will show by using Lemma~\ref{warnke} that $Y$ is tightly concentrated around its expectation. This will imply that a.a.s.\ $Y \ge (1+o(1))\E K$. We will have to estimate $| Y(W_1,\dots,W_i,\dots,W_n)-Y(W_1,\dots,W_i',\dots,W_n) |$ for all $W,i,W_i'$.

Now we will check the assumptions of Lemma~\ref{warnke}. For a given value of $Y=y$ let $P$ be a loose path of length~$y$. Suppose first that $i\in P$ and that edge $e\in (P\cap W_i)\setminus W_i'$. By removing edge $e$ we split $P$ into two paths $P_1$ and $P_2$ such that $V(P_1) \subseteq [i-x]$ and $V(P_2) \subseteq \{i-x+r-1,\dots,n\}$ for some $0\leq x\leq r-1$. Define $c=3a$, where $a$ is defined in Claim~\ref{claim:connector}. 

Suppose first that $|E(P_1)|\ge c$ and $|E(P_2)|\ge c$. 
Let $m_i = |E(P_i)|$, where clearly $m_1+m_2=y-1$. Let $E(P_i) = \{e_1^i,\dots,e_{m_i}^i\}$, where $e_j^i$ precedes $e_{j+1}^i$. Denote by $first(e)$ and $last(e)$ the first and the last vertex of edge~$e$, respectively. Thus, $last(e_j^i) \leq first(e_{j+1}^i)$.
Define
\[
A_1 = \{last(e_j^1): m_1-2a+1\le j\le m_1-a\} \quad \text{and} \quad A_2=\{first(e_j^2):1\le j\le a\}
\]
and $A_3$ be any subset of $\bigcup_{j=m_1-a+1}^{m_1} e_j^1$ of size $a$ which does not contain $last(e^1_{m_1-a})$. Due to Claim~\ref{claim:connector} there exists an edge $f$ such that $f=\{last(e_{j_1}^1)<\dots<first(e_{j_2}^2)\}$ with $m_1-2a+1\le j_1\le m_1-a$ and $1\le j_2\le a$. Thus, $P_1$ and $P_2$ together with $f$ contain a path of length at least $y-c$
yielding $|Y(W_1,\dots,W_i,\dots,W_n)-Y(W_1,\dots,W_i',\dots,W_n) | \le c$. 

Suppose now  that $|E(P_1)| \le c$ or $|E(P_2)|\le c$. Now $P_1$ or $P_2$ have length at least $y-c$. If there is a path $Q$ when $W'$ is added of length $z$ and $Q_1,Q_2$ are obtained by deleting some $e'\in (Q\cap W_i')\setminus W_i$ then either one of $Q_1,Q_2$ has at least $z-c$ edges or we can use the argument for large $P_1,P_2$ to show that $y\geq z-c$. Thus in all cases $|Y(W_1,\dots,W_i,\dots,W_n)-Y(W_1,\dots,W_i',\dots,W_n) | \le c$.

It remains to consider the case in which $i$ is not in $P$. We may assume that $W_i'$ consists of all possible edges. By adding these edges we can create a longer path. Clearly, here we can argue by the symmetric argument that we cannot create a path with length greater than $K+c$.

Now we apply Lemma~\ref{warnke} with $\Gamma$ equal to the set of $W$ satisfying the conditions in Claim~\ref{claim:connector} and $c_i=c$ for $W\in\Gamma$ and $d_i=n$ and $\gamma_i=n^{-4}$, where  
\[
c_i=c=3a=
\begin{cases}
\frac{12}{p} \log n & \text{ if } r=2,\\
3\left(\frac{4(r-2)^{r-2}}{p} \log n \right)^{1/(r-1)} & \text{ if } r\ge3. 
\end{cases}
\]
Thus, for $p\ge\omega\sqrt{\log n}/n^{(r-1)/4}$ (where $\omega=\omega(n)$ represents a function that tends to infinity arbitrarily
slowly) and $t=(\E Y)/\omega^{1/(r-1)}$ we get
\begin{align*}
\Pr\brac{|Y -\E Y| \ge \frac{\E Y}{\omega^{1/(r-1)}}} &\le 2 \exp\left\{-\frac{(\E Y)^2}{2\omega^{2/(r-1)} (n c^2+3cn^{-2})}\right\}+O(n^{-12})\\
&\le 2 \exp\left\{-\frac{(\E K)^2}{2\omega^{2/(r-1)} (n c^2+3cn^{-2})}\right\}+O(n^{-12}).
\end{align*}
Now notice that by Lemma~\ref{lem:geometric}\eqref{prop:geometric:iii} the order of magnitude of the exponent satisfies
\[
\frac{(\E K)^2}{3\omega^{2/(r-1)} n c^2}
\ge \frac{n}{3\omega^{2/(r-1)} c^2 (\E X)^2}
= \frac{n}{3\omega^{2/(r-1)} \left(\frac{\log n}{p}\right)^{2/(r-1)} \left(\frac{1}{p}\right)^{2/(r-1)}}
= \frac{np^{4/(r-1)}}{3\omega^{2/(r-1)} (\log n)^{2/(r-1)}}.
\]
Recall that $\E Y \ge \E K$.
Hence, for $p\ge\omega\sqrt{\log n}/n^{(r-1)/4}$, 
\[
\frac{(\E Y)^2}{\omega^{2/(r-1)} n c^2} = \Omega(\omega^{2/(r-1)})
\]
implying that $\Pr(|Y -\E Y| \ge (\E Y)/\omega^{1/(r-1)})=o(1)$  and so a.a.s.\ $Y\sim \E Y$.

\section{Lower bound for $\Omega(1/n^{(r-1)/2}) = p = O(\sqrt{\log n}/n^{(r-1)/4})$.}\label{sec:4}

Let $\ell=O(n(p/\log n)^{1/(r-1)})$, where the hidden constant is sufficiently small. We will greedily build a path of length $\ell$ by applying Claim~\ref{claim:length_e} (below). Observe that 
\[
d: = \frac{n}{\ell} = \Omega\left( \frac{n}{n(p/\log n)^{1/(r-1)}} \right) =\Omega\left(\left(\frac{\log n}{p} \right)^{1/(r-1)}\right).
\]

\begin{claim}\label{claim:length_e}
Let $p=\Omega((\log n)/n^{r-1})$ and $(r-1)\left(\frac{2\log n}{p} \right)^{1/(r-1)} \le d \le n$. Then, a.a.s.\ for each $i\in[n-d]$ there is an edge~$e\in E_p$ of length at most $d$ that starts at $i$. This means there is a vertex $j\le i+d$ such $e=\{i<i_2<\dots<i_{r-1}<j\}\in E_p$.
\end{claim}
\begin{proof}
Observe that for a fixed vertex $i$, the probability that there is no edge of the form $\{i<i_2<\dots<i_{r-1}<i_r\}$ with $i_r\le i+d$ is exactly $(1-p)^{\binom{d}{r-1}}$. Hence, the probability that there exists a vertex $i$ with no such edges is at most
\[
n (1-p)^{\binom{d}{r-1}} \le \exp\left\{\log n - p\binom{d}{r-1}\right\}
\le \exp\left\{\log n - p\left(\frac{d}{r-1}\right)^{r-1}\right\} = o(1).
\]
\end{proof}
Now we explain our greedy algorithm. First we reveal all $r$-tuples starting at vertex $1$. By Claim~\ref{claim:length_e} there is an edge $e_1$ of length at most $d$ starting at this vertex. Now we reveal all $r$-tuples that start at the last vertex of $e_1$, etc. This way we see that a.a.s.\ there is a path of length at least~$\ell$. This establishes the lower bound in Theorem~\ref{thm:main3}.

\section{Lower bound for $\Omega((\log n)^{r-1}/n^{r-1}) = p \ll 1/n^{(r-1)/2}$.}\label{sec:second_moment}\label{sec:5}

Let $\ell =  \frac{1}{8(r-1)}np^{1/(r-1)}-\frac{1}{r-1}$.
We will apply the second moment method. For a fixed $S\in \binom{[n]}{\ell(r-1)+1}$, let $X_S$ be an indicator random variable which equals 1 if $S$ induces an ordered path of length~$\ell$ denoted by $P_S$; otherwise 0. Define $X=\sum_{S\in \binom{[n]}{\ell(r-1)+1}} X_S$.
Clearly,
\begin{align*}
\E X &= \binom{n}{\ell(r-1)+1}p^\ell
\ge \left(\frac{n}{\ell(r-1)+1}\right)^{\ell(r-1)+1}p^\ell\\
&\ge \left(\frac{n}{\ell(r-1)+1}\right)^{\ell(r-1)}p^\ell
= \left(\left(\frac{n}{\ell(r-1)+1}\right)^{r-1}p\right)^{\ell} = 8^{\ell(r-1)},
\end{align*}
which tends to infinity.

Now we calculate $\sum_{S\neq S'} \E (X_SX_S')$ for $s:=\ell(r-1)+1=|S|=|S'|$. First observe that 
\[
\sum_{|S \cap S'|=0} \E (X_SX_{S'}) \le \binom{n}{s}p^\ell \binom{n}{s}p^\ell = (\E X)^2.
\]
Similarly, if $|S \cap S'|\ge 1$ and $P_S$ and $P_{S'}$ share no edges,
\begin{align*}
\sum_{|S \cap S'|\ge 1} \E (X_SX_S') 
&\le \binom{n}{s}p^\ell \binom{n}{s-1}p^\ell =\binom{n}{s}p^\ell \binom{n}{s}\frac{s}{n-s+1}  p^\ell \\
& = (\E X)^2 \frac{s}{n-s+1} = o((\E X)^2).
\end{align*}

It remains to consider the case when the paths induced by $S$ and $S'$ share some edges. Let $n_{a,b}$ be the number of ordered pairs $(S,S')$ such that $|E(P_S) \cap E(P_{S'})|=a$ and the number of vertices of degree 2 induced by $E(P_S) \cap E(P_{S'})$ is exactly~$b$. For example, if $E(P_S) \cap E(P_{S'})$ is just one path, then $b=a-1$, or if $E(P_S) \cap E(P_{S'})$ is a matching of size~$a$, then~$b=0$. In general, it is easy to observe that $E(P_S) \cap E(P_{S'})$ consists of $c:=a-b$ vertex disjoint paths. 

First we will consider the following problem. Let $P$ be a path of length $\ell$. We calculate in how many ways we can choose vertex-disjoint paths $P_1,\dots,P_c$ contained in~$P$ such that $|E(P_1)|+\dots+|E(P_c)|=a$ (this will also imply that the number of vertices of degree 2 in $P_i$'s is exactly $b$). Let $x_i=|E(P_i)|$ and $y_i$ be a gap (in terms of the number of edges) between $P_i$ and $P_{i+1}$. Any choice of $P_i$'s can be encoded as an integer solution of 
\[
x_1+\dots+x_c = a \quad\text{ and }\quad y_0+\dots+y_c=\ell-a,
\]
where $x_i\ge 1$ for $1\le i\le c$ and $y_0\ge 0$, $y_c\ge 0$, $y_j\ge 1$ for $1\le j\le c-1$. This is equivalent to solving
\[
(x_1-1)+\dots+(x_c-1) = a-c 
\]
and
\[
\ y_0+(y_1-1)+(y_2-1)+\dots+(y_{c-1}-1)+y_c=\ell-a-c+1
\]
with all $x_{1}-1, \ldots, x_{c}-1, y_{0}, y_{1}-1, \ldots, y_{c-1}- 1, y_{c}$ nonnegative.
Thus, the number of integer solutions is exactly
\[
\binom{a-1}{c-1} \binom{\ell-a+1}{c} \le 2^a \ell^{c} = 2^a \ell^{a-b}.
\]

Now notice that as $s=\ell(r-1)+1$ and $|S \cap S'|=a r-b$,
\begin{align*}
n_{a,b} &\le \binom{n}{s} \cdot 2^a \ell^{a-b}\binom{n}{s-(ar-b)}\\
&= \binom{n}{s} \cdot 2^a \ell^{a-b}\binom{n}{s} \frac{s-(ar-b)+1}{n-s+1}\cdots \frac{s-1}{n-s+(ar-b-1)}
\frac{s}{n-s+(ar-b)}\\
&\le \binom{n}{s} \binom{n}{s} 2^a \ell^{a-b}\left(\frac{s}{n-s}\right)^{ar-b}.
\end{align*}
Hence,
\begin{align*}
\sum_{1\le a \le \ell}\; \sum_{0\le b \le a-1} n_{a,b} p^{2\ell-a}
&\le \sum_{1\le a \le \ell}\; \sum_{0\le b \le a-1} (\E X)^2  2^a\ell^{a-b} \left(\frac{s}{n-s}\right)^{ar-b} p^{-a}\\
&= (\E X)^2 \sum_{1\le a \le \ell}\; 2^a\ell^{a} \left(\frac{s}{n-s}\right)^{ar} p^{-a} \sum_{0\le b \le a-1}  \ell^{-b} \left(\frac{s}{n-s}\right)^{-b}. 
\end{align*}

Since $p \ll 1/n^{(r-1)/2}$, it follows that $\ell^2 \ll n$ and so $\ell s \ll n-s$. Thus, $\frac{n-s}{\ell s}-1 \ge \frac{n}{2\ell s}$ and
\[
\sum_{0\le b \le a-1}  \ell^{-b} \left(\frac{s}{n-s}\right)^{-b}
= \sum_{0\le b \le a-1}   \left(\frac{n-s}{\ell s}\right)^{b}
= \frac{\left(\frac{n-s}{\ell s}\right)^{a}-1}{\frac{n-s}{\ell s}-1}
\le \frac{\left(\frac{n-s}{\ell s}\right)^{a}}{\frac{n}{2\ell s}}.
\]
Hence,
\begin{align}
\sum_{1\le a \le \ell}\; \sum_{0\le b \le a-1} n_{a,b} p^{2\ell-a}
&\le (\E X)^2 \frac{2\ell s}{n} \sum_{1\le a \le \ell}\; 2^a\ell^{a} \left(\frac{s}{n-s}\right)^{ar} p^{-a} \left(\frac{n-s}{\ell s}\right)^{a}\notag\\
&= (\E X)^2 \frac{2\ell s}{n} \sum_{1\le a \le \ell}\; \left( 2\left(\frac{s}{n-s}\right)^{r-1} p^{-1} \right)^a.\label{eq:second_moment}
\end{align}

Recall that $s=\ell(r-1)+1 = \frac{1}{8(r-1)} np^{1/(r-1)}$ and $s\ll n$. Thus, 
\[
2\left(\frac{s}{n-s}\right)^{r-1} p^{-1} 
\le 2\left(\frac{2s}{n}\right)^{r-1} p^{-1} 
= 2\left(\frac{1}{4(r-1)}\right)^{r-1} \le 1/2
\]
and
\[
\sum_{1\le a \le \ell}\; \left(2 \left(\frac{s}{n-s}\right)^{r-1} p^{-1} \right)^a 
\le \sum_{1\le a \le \ell}\; 1/2^a \le 1.
\]
Consequently,
\[
\sum_{1\le a \le \ell}\; \sum_{0\le b \le a-1} n_{a,b} p^{2\ell-a} \le (\E X)^2 \frac{2\ell s}{n} = o((\E X)^2).
\]

Summarizing, we have shown that 
\[
\frac{\E X^2}{(\E X)^2} = 1+o(1).
\]
Therefore, by the second moment method, if $\ell =  \frac{1}{8(r-1)}np^{1/(r-1)}-\frac{1}{r-1}$, then a.a.s.\ there is a path of length~$\ell$. This establishes the lower bound in Theorem~\ref{thm:main2} for the lower range of~$p$.

\section{Upper bound for $p=\Omega((\log n)^{r-1}/n^{r-1})$.}\label{sec:6}

Here we present two general upper bounds. The first one is better in the case when $p$ is bounded from below by a constant -- implying the bound in Theorem~\ref{thm:main1}. The second approach gives the upper bounds in Theorems~\ref{thm:main2} and~\ref{thm:main3}.
Combining the results of this section with those in Sections~\ref{sec:3}-\ref{sec:5} completes the proofs of Theorems~\ref{thm:main1}, \ref{thm:main2} and~\ref{thm:main3}.

\subsection{Bounding the union of paths}
Let $P$ be an ordered loose path with $\ell$ edges.  For an integer $i$ satisfying $r-1\le i \le n-1$, let $x_i$ count the number of edges of length of $i$ in~$P$. Clearly, $\ell= x_{r-1} +x_r + \dots  + x_{n-1}$. Observe that
\[
n-1 \ge (r-1)x_{r-1}+rx_{r}+\dots+(n-1)x_{n-1} \ge (r-1)x_{r-1} + r(\ell-x_{r-1}).
\]
Thus,
\begin{equation}\label{ell_ub}
\ell \le \frac{n-1 + x_{r-1}}{r}.
\end{equation}

Now we will bound $x_{r-1}$ from above by using a greedy approach. Let $H$ be an $r$-uniform hypergraph on the set of vertices $[n]$ (not necessary random) and $P$ an ordered path.
Let $Q$ be a subgraph of $P$ that consists of edges of length~$r-1$ only. That means $Q$ is a union of loose paths and  $|E(Q)|=x_{r-1}$. Assume that there are edges $e\in E(H)\setminus E(Q)$ with $\ell(e)=r-1$ and $f\in E(Q)$ such that $e$ precedes before $f$
and $\left(E(Q)\setminus f\right) \cup e$ is still a union of loose paths. 
That is, $f$ is the first edge of a path in~$Q$, and $first(e)< first(f)$. (Moreover, the lengths of $e$ and $f$ also imply that $last(e)< last(f)$.) Now we replace $f$ by $e$ obtaining a new union of loose paths $R$ of length~$x_{r-1}$. We can repeat this process until such $e$ and $f$ no longer exist and assume that $R$ has this property. Now we show that a greedy algorithm finds $R$.\\

Let $F=\emptyset$. Let $f_1$ be the first present in~$H$ edge of length $r-1$ for which $first(f_1)$ is as small as possible. If such an edge exists, then we add it to $F$; otherwise, we terminate. Assume that we already have chosen $F=\{f_1,\dots,f_i\}$ with $last(f_1)<last(f_2)<\dots<last(f_i)$ that creates a union of loose paths. Next we add to $F$ the next available edge $f_{i+1}$ of length~$r-1$ the occurs after $f_i$, i.e., $\ell(f_{i+1})=r-1$, $last(f_i)\le first(f_{i+1})$ and we repeat the process. If such $f_{i+1}$ does not exist, we terminate. Observe that at the end $F = E(R)$.

Let $Y$ be the random variable that counts the number of edges in the largest union of loose paths in $H^{(r)}(n, p)$ containing only edges of length $r-1$. We will show that $Y$ is concentrated around its mean. Lemma~\ref{lem:x_r_1} will imply that a.a.s.\
\begin{equation}\label{xr1_ub}
x_{r-1} \le Y = (1+o(1))n/(r-2 + p^{-1}).
\end{equation}

Now we show by using McDiarmid's inequality\footnote{This can be viewed as Warnke's inequality (from Lemma~\ref{warnke}) with $\Gamma=\Omega$. See Section~\ref{sec:3} for more details.} that $Y\sim\E Y$. Choose 
\[
t=\frac{\E Y}{\omega} = \frac{n}{\omega(r-2 + p^{-1})} \quad { and } \quad c_i=2,
\]
since by removing one vertex (or adding), we change the size of a union of loose paths by at most two. Thus,
\[
\Pr(|Y-\E Y| \ge t) \le 2 \exp\left\{-\frac{t^2}{2\sum_{i=1}^n c_i^2}\right\}
= 2 \exp\left\{-\frac{n}{8\omega^2 (r-2 + p^{-1})^2}\right\}.
\]
This probability is $o(1)$ when $p=\Omega(1)$ and $\omega = o(\sqrt{n})$.
\begin{lemma}\label{lem:x_r_1}
Let $H$ be a random $r$-hypergraph on the vertex set $[n]$ such that each edge $\{j+1<j+2<\dots<j+r\}$ (for $j\ge0$) appears with probability~$p$. Then,
\[
\E Y\leq \frac{n+p^{-1}}{r-2 +p^{-1}}.
\]
\end{lemma}
\begin{proof}
We prove this statement by induction on $n\ge r$. This is true for $n=r$, since $(r+p^{-1})/(r-2+p^{-1})\geq 1$. We then write
\begin{equation*}
\E Y\leq 1+\sum_{i=0}^{n-2(r-1)-1}(1-p)^ip\frac{n-(r-1+i)+p^{-1}}{r-2+p^{-1}}.
\end{equation*}
Here we have 1 for the first edge of our union. This ends at $r+i$ with probability $(1-p)^ip$ and the remaining term follows from our inductive assumption. Continuing, by increasing the upper index of summation, we get 
\[
\E Y \leq 1+\frac{p}{r-2+p^{-1}}\sum_{i=0}^{n-(r-1)}(1-p)^i(n-(r-1+i)+p^{-1}).
\]
Since
\begin{align*}
\sum_{i=0}^{n-(r-1)}(1-p)^i&(n-(r-1+i)+p^{-1})\\
&=(n-(r-1)+p^{-1}) \sum_{i=0}^{n-(r-1)}(1-p)^i - \sum_{i=0}^{n-(r-1)}i(1-p)^i\\
&=(n-(r-1)+p^{-1})\frac{1-(1-p)^{n-r+2}}{p}\\ 
&\qquad\qquad+ \frac{(n-(r-1)+p^{-1})(1-p)^{n-r+2}+1-p^{-1}}{p}
=\frac{n-r+2}{p},
\end{align*}
we get that
\[
\E Y\leq 1+\frac{p}{r-2+p^{-1}}\cdot \frac{n-r+2}{p}=\frac{n+p^{-1}}{r-2+p^{-1}},
\]
completing the induction.
\end{proof}
Combining~\eqref{ell_ub} and~\eqref{xr1_ub} proves the upper bound from Theorem~\ref{thm:main1}.

\subsection{First moment method}\label{sec:upper_bound}
Let $X$ count the number of paths of length~$\ell =  \frac{4^{1/(r-1)}e}{r-1}np^{1/(r-1)}$. If $p \ge (\frac{r-1}{4^{1/(r-1)}e})^{r-1}(\log n)^{r-1}/n^{r-1}$, then $\ell  \ge \log n$ and
\begin{align*}
\E X &= \binom{n}{\ell(r-1)+1}p^\ell
\le \left(\frac{en}{\ell(r-1)}\right)^{\ell(r-1)+1}p^\ell\\
&= \left(\left(\frac{en}{\ell(r-1)}\right)^{r-1}p  \left(\frac{en}{\ell(r-1)}\right)^{1/\ell} \right)^\ell\\
&\le \left(\left(\frac{en}{\ell(r-1)}\right)^{r-1}p  n^{1/\ell} \right)^\ell
= \left(\frac{1}{4} n^{1/\ell}\right)^{\ell}\\
&\le \left(\frac{1}{4} n^{1/\log n}\right)^{\ell}
= \left(\frac{e}{4}\right)^{\ell} = o(1).
\end{align*}
Thus, if $\ell \ge  \frac{4^{1/(r-1)}e}{r-1}np^{1/(r-1)}$, then a.a.s.\  there is no path of length $\ell$.
This establishes the upper bound in Theorems~\ref{thm:main2} and~\ref{thm:main3}.

\section{Lower and upper bound for $\frac{1}{n^{r-1+1/\omega}} \le p \le O((\log n)^{r-1}/n^{r-1})$ with~$1 \ll \omega \ll \log n$.}\label{sec:7}

Here we will prove Theorem~\ref{thm:main4}. The first part of the statement is proved in the following claim.

\begin{claim}
There exists $1\ll \ell_0 = O(\log n)$ such that 
\[
\left(\frac{n}{\ell_0}\right)^{r-1+1/\ell_0} p =1.
\] 
\end{claim}
\begin{proof}
For $x>0$ define
\[
f(x) = \left(\frac{n}{x}\right)^{r-1+1/x} p.
\]
Clearly, $f$ is a continuous function. We will show that there are $1\ll x_1 < x_2 = O(\log n)$ such that $f(x_1)>1$ and $f(x_2)<1$. Thus,  the Intermediate Value Theorem will imply the statement.

Recall that $\frac{1}{n^{r-1+1/\omega}} \le p$, where $1\ll \omega \ll \log n$. Let $x_1=\log \omega$. Note that
\[
f(x_1) = \left(\frac{n}{x_1}\right)^{r-1+1/x_1} p
\ge \left(\frac{n}{\log \omega}\right)^{r-1+1/\log \omega} \frac{1}{n^{r-1+1/\omega}}
= \frac{n^{1/\log \omega - 1/\omega}}{(\log \omega)^{r-1+1/\omega}}
\ge \frac{n^{1/\log \omega - 1/\omega}}{(\log \omega)^{r}}.
\]
For large $n$, we have 
\[
1/\log \omega - 1/\omega \ge 1/(2\log \omega) \ge 1/(2\log \log n).
\]
Thus, when $n$ is sufficiently large
\[
f(x_1) \ge \frac{n^{1/(2\log \log n)}}{(\log \log n)^{r}}
= \exp\left\{ \frac{\log n}{2\log \log n} - r\log \log \log n \right\},
\]
which tends to infinity. Hence, $f(x_1) > 1$.

Now let $C>0$ be an arbitrarily large constant and assume that $p\le \frac{C(\log n)^{r-1}}{n^{r-1}}$. Let $d>1$ be (sufficiently large) such that 
\[
\frac{Ce^{1/d}}{d^{r-1}}   < 1.
\]
Define $x_2 = d \log n$ and observe that
\begin{align*}
f(x_2) &\le \left(\frac{n}{d\log n}\right)^{r-1+1/(d\log n)} \frac{C(\log n)^{r-1}}{n^{r-1}}\\
&= \frac{C}{d^{r-1+1/(d\log n)}} \left(\frac{n}{\log n}\right)^{1/(d\log n)} 
\le \frac{C}{d^{r-1}} n^{1/(d\log n)} = \frac{Ce^{1/d}}{d^{r-1}}   < 1.
\end{align*}
\end{proof}

\subsection{Upper bound}
Let $\ell =  \frac{2e(r-1)\ell_0-1}{r-1}$. In the next two sections, $X$~denotes the number of paths of length $\ell$, for the most recently stated value of $\ell$. Hence, $1/\ell \le 1/\ell_0$ and, using the definition of $\ell_{0}$
\begin{align*}
\E X &= \binom{n}{\ell(r-1)+1}p^\ell
\le \left(\frac{en}{\ell(r-1)+1}\right)^{\ell(r-1)+1}p^\ell\\
&= \left(\left(\frac{n}{2(r-1)\ell_0}\right)^{r-1+1/\ell}p\right)^\ell
\le \left(\left(\frac{n}{2(r-1)\ell_0}\right)^{r-1+1/\ell_0}p\right)^\ell\\
&= \left(\frac{1}{(2(r-1))^{r-1+1/\ell_0}} \left(\frac{n}{\ell_0}\right)^{r-1+1/\ell_0}p\right)^\ell
= \frac{1}{(2(r-1))^{\ell(r-1)+\ell/\ell_0}} = o(1).
\end{align*}

\subsection{Lower bound}
Let $\ell =  \frac{\ell_0/4 - 1}{r-1}$. Hence, $1/\ell \ge 1/\ell_0$ and 
\begin{align*}
\E X &= \binom{n}{\ell(r-1)+1}p^\ell
\ge \left(\frac{n}{\ell(r-1)+1}\right)^{\ell(r-1)+1}p^\ell\\
&= \left(\left(\frac{4n}{\ell_0}\right)^{r-1+1/\ell}p\right)^\ell
\ge \left(\left(\frac{4n}{\ell_0}\right)^{r-1+1/\ell_0}p\right)^\ell\\
&= \left(4^{r-1+1/\ell_0} \left(\frac{n}{\ell_0}\right)^{r-1+1/\ell_0}p\right)^\ell
= 4^{\ell(r-1)+\ell/\ell_0} \gg 1.
\end{align*}

We will now modify the calculations from Section~\ref{sec:second_moment}. Since $s=\Theta(\ell) = O(\log n)$, we have
\begin{align*}
2\left(\frac{s}{n-s}\right)^{r-1} p^{-1} 
&\le 2\left(\frac{2s}{n}\right)^{r-1} p^{-1}
= 2\left(\frac{2(\ell(r-1)+1)}{n}\right)^{r-1} p^{-1}\\
&= 2\left(\frac{\ell_0}{2n}\right)^{r-1} p^{-1}
= \frac{1}{2^{r-2}}\left(\frac{n}{\ell_0}\right)^{1/\ell_0}
\le n^{1/\ell_0}.
\end{align*}
Consequently, we bound~\eqref{eq:second_moment} as follows:
\begin{align*}
\sum_{1\le a \le \ell}\; \sum_{0\le b \le a-1} n_{a,b} p^{2\ell-a}
&\le (\E X)^2 \frac{2\ell s}{n} \sum_{1\le a \le \ell}\; \left(2 \left(\frac{s}{n-s}\right)^{r-1} p^{-1} \right)^a\\
&\le (\E X)^2 \frac{2\ell s}{n} \sum_{1\le a \le \ell}\; \left( n^{1/\ell_0} \right)^a
\le (\E X)^2 \frac{\ell s}{n} O\left( n^{1/\ell_0} \right)^\ell.
\end{align*}
Since $\ell/\ell_0 < 1/(4(r-1))$ and $\ell s = O(\log^2 n)$, we get that 
\begin{align*}
\sum_{1\le a \le \ell}\; \sum_{0\le b \le a-1} n_{a,b} p^{2\ell-a}
&\le (\E X)^2 \frac{\ell s}{n} O(n^{\ell/\ell_0})\\
&\le (\E X)^2 \frac{\ell s}{n}O(n^{1/(4(r-1))})\\
&= (\E X)^2 O\left(\frac{\log^2 n}{n^{(4r-5)/(4r-4)}}\right)
= o((\E X)^2),
\end{align*}
As in Section~\ref{sec:5} we conclude that a.a.s.\ there is an ordered path of length $\ell$.

\section{Lower and upper bound for $1/(n^{r-1+1/\ell}) \ll p \ll 1/(n^{r-1+1/(\ell+1)})$ where $\ell=O(1)$.}\label{sec:8a}

\subsection{Upper bound}\label{sec:8}

Let $X$ count the number of paths of length~$\ell+1$.
Recall that from Section~\ref{sec:upper_bound} we know that
\[
\E X = \binom{n}{(\ell+1)(r-1)+1}p^{\ell+1}
\le n^{(\ell+1)(r-1)+1} p^{\ell+1}
\ll n^{(\ell+1)(r-1)+1} \cdot 1/(n^{(\ell+1)(r-1)+1}) = 1
\]
implying that $\E X$ tends to 0.

\subsection{Lower bound}

Similarly, for $X$ counting the number of paths of length~$\ell$,
\[
\E X = \binom{n}{\ell(r-1)+1}p^\ell
\gg \left(\frac{n}{\ell(r-1)+1}\right)^{\ell(r-1)+1} \cdot 1/(n^{\ell(r-1)+1}) = \Omega(1)
\]
yielding that $\E X$ tends to infinity.

We will now modify calculations from Section~\ref{sec:second_moment}. By assumption we have $\ell=O(1)$ and $s=\ell(r-1)+1=O(1)$. Observe that in~\eqref{eq:second_moment} we can use the following upper bounds,
\[
2\left(\frac{s}{n-s}\right)^{r-1} p^{-1} \ll 1/n^{r-1} \cdot n^{r-1+1/\ell} = n^{1/\ell}.
\]
Hence,
\begin{align*}
\sum_{1\le a \le \ell}\; \sum_{0\le b \le a-1} n_{a,b} p^{2\ell-a}
&\le (\E X)^2 \frac{2\ell s}{n} \sum_{1\le a \le \ell}\; \left(2 \left(\frac{s}{n-s}\right)^{r-1} p^{-1} \right)^a\\
&\ll (\E X)^2 \frac{O(1)}{n} O(n) = (\E X)^2 O(1),
\end{align*}
showing that 
\[
\sum_{1\le a \le \ell}\; \sum_{0\le b \le a-1} n_{a,b} p^{2\ell-a} =o((\E X)^2),
\]
as required.

\subsection{Method of moments}
Observe that if $p=c/(n^{r-1+1/\ell})$, then
\[
\E X = \binom{n}{\ell(r-1)+1}p^\ell
\sim \frac{n^{\ell(r-1)+1} }{(\ell(r-1)+1)!}\cdot c^\ell/(n^{\ell(r-1)+1}) = c^\ell/(\ell(r-1)+1)! =:\lambda.
\]
Now by generalizing the second moment calculations one can show that 
\[
\E X(X-1)\cdots (X-k+1)\to \lambda^k
\] 
for any $k\ge1$. Thus, the method of moments (see, e.g.\ Section~29.2 in~\cite{FK}) yields the final statement of Theorem~\ref{thm:main5}. 
\section{Concluding remarks}
We have found high probability estimates for $\ell_{\max}$ for all ranges of $p$. The gaps between upper and lower bounds are all of order $O(1)$ except for Theorem~\ref{thm:main3} where $\Omega(1/n^{(r-1)/2}) = p = O(\sqrt{\log n}/n^{(r-1)/4})$. In this case we are off by an order $O((\log n)^{1/(r-1)})$ factor. It would be of some interest to remove this.

\bigskip
\textbf{Acknowledgment.}
We would like to express our sincere gratitude to the referees for their careful reading of our manuscript and for their insightful comments and constructive suggestions.

\end{document}